\documentclass[12pt]{amsart}%
\usepackage{graphicx}
\usepackage{epsfig}
\usepackage{amscd}
\usepackage{amsmath}
\usepackage{amsfonts}
\usepackage{amssymb}
\usepackage[latin1]{inputenc}%
\usepackage[colorlinks=true, allcolors=blue]{hyperref}
\usepackage{indentfirst}
\providecommand{\U}[1]{\protect\rule{.1in}{.1in}}
\newtheorem{theorem}{Theorem}[section]

\newtheorem{corollary}[theorem]{Corollary}

\newtheorem{lemma}[theorem]{Lemma}
\newtheorem{notation}[theorem]{Notation}

\newtheorem{proposition}[theorem]{Proposition}
\newtheorem{remark}[theorem]{Remark}

\newcommand{\supp}{\text{\rm Supp}}

\begin{document}

\title{Identities and isomorphisms of finite graded matrix algebras}

\author[Silva]{Diogo Diniz}
\thanks{D. Diniz was partially supported by FAPESQ grant No.~3099/2021 and CNPq grant No.~304328/2022-7}
\address{Universidade Federal de Campina Grande (UFCG) \\
	Departamento de Matem\'{a}tica \\
	58429-970 Campina Grande, PB, Brazil}
\email{diogo@mat.ufcg.edu.br}

\author[Martinez Correa]{Daniela Martinez Correa}

\thanks{D. Martinez Correa was supported by FAPESP, grant no.~2024/01338-0 and 2025/11605-8.}
\address{Universidade de S\~ao Paulo\\ Instituto de Matem\'atica e Estat\'istica e Ci\^encia da Computa\c{c}\~ao, 05508-090 S\~ao Paulo, SP, Brazil}
\email{danielam.correa@ime.usp.br}

\author[Gon\c{c}alves]{Dimas Jos\'e Gon\c{c}alves}
\thanks{D. J. Gon\c{c}alves was partially supported by FAPESP grant No.~2024/14914-9}
\address{Universidade Federal de S\~ao Carlos (UFSCar) \\
Departamento de Matem\'{a}tica \\
13565-905 S\~ao Carlos, SP, Brazil}
\email{dimasjg@ufscar.br}

\author[Koshlukov]{Plamen Koshlukov}
\thanks{P. Koshlukov was partially supported by FAPESP grant No.~2024/14914-9 and by CNPq grant No.~307184/2023-4}
\address{Universidade Estadual de Campinas (UNICAMP) \\
Departamento de Matem\'{a}tica \\
13083-859 Campinas, SP,  Brazil}
\email{plamen@unicamp.br}

	\date{}

    \begin{abstract}
    Let $\mathbb{F}$ be a finite field and let $G$ be a group such that every finite subgroup of $G$ is cyclic. In this paper we classify the division gradings by cyclic groups on matrix algebras over $\mathbb{F}$. This yields the classification of the $G$-gradings on these algebras. Moreover if $G$ is abelian we prove that two matrix algebras over $\mathbb{F}$ with $G$-gradings satisfy the same graded identities if and only if they are isomorphic as graded algebras. 
    \end{abstract}


	
	\maketitle

\section{Introduction}
 
The polynomial identities of simple algebras play an important role in the study of algebraic structures. A fundamental result in this direction is the well known Amitsur--Levitzki Theorem that says that the standard polynomial of degree $2n$ is an identity of minimal degree for the algebra of $n\times n$ matrices over a field. An interesting application of this result can be found in \cite{BGGT}. An immediate consequence of the Amitsur--Levitzki Theorem is that finite dimensional simple algebras, over an algebraically closed field, are determined up to isomorphism by their polynomial identities. If the field is not algebraically closed then the result may not hold, for example, the real algebras $\mathbb{H}$ of the quaternions and $M_2(\mathbb{R})$ satisfy the same polynomial identities but are not isomorphic. If the base field is finite, however, the finite dimensional simple algebras are determined by their polynomial identities, see Proposition \ref{ordinarymain} in Section \ref{prelim}.

Graded algebras and graded polynomial identities play a fundamental role in Kemer's Representability Theorem which was a key step in his solution to the Specht problem. In particular the structure theory of finite-dimensional graded algebras and  graded simple algebras plays an important role in the Representability Theorem. In this context a natural question arises: whether the graded identities determine the finite dimensional graded simple algebras up to isomorphism. The answer to this question is affirmative if the base field is algebraically closed, see \cite{AH}, \cite{K} and \cite{KZ}. 
 The analogous question for upper block-triangular matrix algebras was studied in \cite{RD}. This problem was solved for non-associative algebras and other algebraic structures, see \cite{BY}, \cite{BY2} and \cite{BD}.

The polynomial identities of $n\times n$ matrix algebras have been described for $n=2$ only if the base field is infinite and of characteristic different from $2$, however over a finite field the description has been made for $n=2,3,4$, we refer the reader to the papers \cite{G}, \cite{GP}, \cite{Ko} and \cite{MK} for more details. The graded polynomial identities for matrix algebras of order two over a finite field were described in \cite{KS}. In the present paper we prove that the gradings by a group $G$ on matrix algebras over a finite field are determined up to isomorphism by their graded identities provided that $G$ is abelian and every finite subgroup of $G$ is cyclic. Here we recall that the latter restriction is not artificial. For example, such is the multiplicative group of every field. Clearly if $G$ is such a group then every subgroup of $G$ satisfies the above property. On the other hand direct products of such groups, and also homomorphic images, may fail to have the same property. 
The description of the finite groups such that all of their proper subgroups are cyclic is contained in the 1903 paper \cite{miller_moreno}. There exist various infinite groups with this property. Recall that Olshanskii \cite{olshanskii} constructed examples of infinite groups such that every subgroup is of finite prime order. In order to prove our result we classify the division gradings on matrix algebras over finite fields whose support is a finite cyclic group. This yields the classification of the $G$-gradings on matrix algebras over finite fields. Over algebraically closed fields the graded division algebras are exactly the twisted group algebras with the canonical grading. The gradings on matrix algebras over an algebraically closed field were classified by Bahturin, Sehgal and Zaicev, see \cite{BSZ} and \cite{BZ}. 

The paper is organized as follows:
in Section \ref{prelim} we present the main notions and results on ordinary and graded polynomial identities that will be used in the paper. In Section \ref{prelimcrossed} we present the results on crossed products that will be needed in the remaining sections. In Section \ref{divisiongradings} we classify the division gradings on matrix algebras that have cyclic support. In Section \ref{example} we present an example of a grading by a cyclic group of order $2$ on $M_2(\mathbb{F})$, where $\mathbb{F}$ is the field with two elements, that makes the classification results in Section \ref{divisiongradings} and the results in the last section of the paper easier to understand. In Section \ref{gradingsmatrix} we present the classification of the $G$-gradings on $M_n(\mathbb{F})$, where $\mathbb{F}$ is a finite field and $G$ is a group for which every finite subgroup is cyclic. In Section \ref{gradedpoly} we obtain the main result of the paper, Theorem \ref{main}, that says that if $G$ is also abelian then the $G$-gradings on matrix algebras in Section \ref{gradingsmatrix} are determined, up to isomorphism, by their graded identities.

 \section{Preliminaries on polynomial identities}\label{prelim}

Let $\mathbb{F}$ be a finite field with $q$ elements and $p=\mathrm{char}\, \mathbb{F}$. Henceforth the algebras we work with are over $\mathbb{F}$. In this section we present the definitions and basic results about polynomial identities, both ordinary and graded, that will be needed in the remaining sections of the paper.

Let $X=\{x_1, x_2,  \ldots\}$ be a countable set of indeterminates and let $\mathbb{F}\langle X\rangle$ be the free associative algebra, freely generated by  $X$, over $\mathbb{F}$. We say that a polynomial $f(x_1,\ldots, x_n)\in \mathbb{F}\langle X\rangle$ is a polynomial identity for an associative algebra $R$, if $f(r_1,\ldots, r_n)=0$ for all $r_1$, \dots, $r_n\in R$. The polynomial identities of $R$ form a $T$-ideal of $\mathbb{F}\langle X\rangle$, that is an ideal closed under all endomorphism of $\mathbb{F}\langle X\rangle$. We denote by $Id(R)$ the T-ideal of polynomial identities of $R$.

We say that a polynomial $f=f(x_1, \ldots,  x_n)\in\mathbb{F}\langle X\rangle$ is a central polynomial for an associative algebra $R$ if $f$ has no constant term, $f(r_1, \ldots,  r_n)$ belongs to the centre of $R$ for all $r_1$, \dots, $r_n\in R$, and $f\notin Id(R)$. It is well known that $f(x_1,x_2)=[x_1,x_2]^2$ is a central polynomial for the algebra of  the $2\times 2$ matrices with entries in $\mathbb{F}$, $M_2(\mathbb{F})$. A natural question in this direction was posed by Kaplansky in 1956: Does there exist a multihomogeneous central polynomial for the matrix algebra $M_k(\mathbb{F})$, $k>2$. It was solved by means of direct constructions provided independently by Razmyslov \cite{Ra} and by Formanek \cite{F}, with the following result.

\begin{theorem}[E. Formanek \cite{F}] \label{teoremadeformanek}
For each integer $n\geq 2$ there exists a polynomial \linebreak$F_n=F_n(x_1,y_1,\dots, y_n)\in \mathbb{Z}\langle x_1, y_1,\dots, y_n\rangle$ with the following properties:
\begin{enumerate}
\item[(a)] $F_n$ is homogeneous of degree $n^2-n$ in $x_1$ and multilinear in $y_1$, \dots, $y_n$.
			
\item[(b)] If $\mathbb{K}$ is a field, $F_n$ is a central polynomial for $M_n(\mathbb{K})$.
\end{enumerate}
\end{theorem}
We will call the polynomial $F_n$ in Theorem \ref{teoremadeformanek} Formanek's polynomial. With respect to item (b), we observe that
\begin{equation}\label{24denovembro1}F_n(M_n(\mathbb{K}))= \{F_n(X_1,Y_1,\dots, Y_n)\mid X_1,Y_1,\dots, Y_n \in M_n(\mathbb{K})\}=\mathbb{K}I_n,
\end{equation}
where $I_n$ is the identity matrix in $M_n(\mathbb{K})$. A very important open problem in PI theory is to determine the least degree of a central polynomial for the matrix algebras $M_n(\mathbb{K})$; we refer to \cite[Section 7.3]{D} for a discussion on the topic, and also for references concerning the applications of central polynomials in PI theory. 

We recall that for a positive integer $m$ the standard polynomial of degree $m$ is the following polynomial \[s_m=s_m(x_1,\ldots ,x_m)=\sum_{\sigma\in Sym(m)}(\mathrm{sign}\, \sigma)x_{\sigma(1)}\cdots x_{\sigma(m)},\]
where $Sym(m)$ is the symmetric group of $\{1,\ldots ,m\}$ and $\mathrm{sign}\, \sigma$ is the sign of the permutation $\sigma\in Sym(m)$. 
	
\begin{theorem}[S. A. Amitsur, J. Levitzki \cite{AL}] \label{teoremaamitsurlevitzki5deoutubro}
Over any field $\mathbb{K}$, the standard polynomial of degree $2n$ is an identity for $M_n(\mathbb{K})$. Moreover, $M_n(\mathbb{K})$ does not satisfy polynomial identities of lower degree. Furthermore, up to a multiplicative constant, $s_{2n}$ is the only multilinear polynomial identity of degree $2n$ for $M_n(\mathbb{K})$, apart from the case when $|\mathbb{K}|=2$ and $n=2$.
\end{theorem}
	
Now, we use these two theorems to prove the ``ordinary" version of our main result (Theorem \ref{main}).
	
\begin{proposition}\label{ordinarymain}
Let $R$ and $R^{\prime}$ be finite simple algebras over $\mathbb{F}$. Then $Id(R)=Id(R^{\prime})$ if and only if $R$ and $R^{\prime}$ are isomorphic as $\mathbb{F}$-algebras.
\end{proposition}
\begin{proof}  
Suppose that $Id(R)=Id(R^{\prime})$. 
By Wedderburn's theorem, there exist a positive integer $n$ and a division algebra $\mathbb{L}$ over $\mathbb{F}$ such that $R \cong M_n(\mathbb{L})$, see \cite[Theorem 2.61]{B}. 
Since $R$ is finite it follows that $\mathbb{L}$ is also finite. Since $\mathbb{L}$ is a finite division ring, Wedderburn's little theorem implies that $\mathbb{L}$ is a field, 
see \cite[Theorem 8.23]{R}.	Therefore $R\cong M_n(\mathbb{L})$ and $R' \cong M_{n^{\prime}}(\mathbb{L}^{\prime})$, where $n$, $n^{\prime}$ are positive integers and $\mathbb{L}$, $\mathbb{L}^{\prime}$ are finite extensions of $\mathbb{F}$. Now, Theorem \ref{teoremaamitsurlevitzki5deoutubro} implies that $n=n^{\prime}$. 

Let $r$, $r^{\prime}$ be positive integers such that $| \mathbb{L} |=q^{r}$ and $| \mathbb{L}^{\prime}|=q^{r^{\prime}}$. We claim that $r=r^{\prime}$. Indeed assume on the contrary that $r\neq r^{\prime}$, and without loss of generality suppose that $r<r^{\prime}$. By Theorem \ref{teoremadeformanek} and (\ref{24denovembro1}),  $F_n(R)=\mathbb{L}I_n$, where $I_n$ is the identity matrix in $M_n(\mathbb{L})$. Hence $F_n^{q^r}-F_{n}$ is a polynomial identity for $R$. Since $F_n(R^{\prime})=\mathbb{L}^{\prime}I_n$ and $r<r^{\prime}$, we conclude that $F_n^{q^r}-F_{n}$ is not an identity for $R^{\prime}$, a contradiction. Hence $r=r^{\prime}$, and this implies that $\mathbb{L}$ and $\mathbb{L}^{\prime}$ are splitting fields of the polynomial $x^{q^r}-x$ over $\mathbb{F}$, therefore \cite[Theorem 3.131]{R} implies that $\mathbb{L}\cong \mathbb{L}^{\prime}$. 
We conclude that $M_n(\mathbb{L})\cong M_n(\mathbb{L}^{\prime})$ and this implies that $R\cong R^{\prime}$. 

The converse is clear, namely if $R\cong R^{\prime}$ then $Id(R)=Id(R^{\prime})$.
\end{proof}

Finally, we recall briefly some basic notions about graded algebras and graded polynomial identities. Let $G$ be a group and consider $R$ an associative algebra. Then $R$ is a $G$-graded algebra whenever $R=\oplus_{g\in G} R_g$ is a direct sum of vector subspaces such that $R_gR_h\subseteq R_{gh}$, for every $g$,  $h\in G$. The elements of $\cup_{g\in G}R_g$ are called homogeneous and a non-zero element of $R_g$ is homogeneous of degree $g$. The support of the $G$-grading on $R$ is defined as the set 
\[
\supp(R) =\{g\in G \mid R_g\neq 0\}.
\] 
The grading $R=\oplus_{g\in G} R_g$  is a division grading if every non-zero homogeneous element is invertible in $R$. A subspace  $I\subseteq R$ is called graded (or homogeneous) if $I=\oplus_{g\in G} I\cap R_g$. 
If $I\subseteq R$ is an ideal and a graded subspace, we call it a graded ideal. In this case, the quotient $R/I$ inherits from $R$ a natural structure of $G$-graded algebra.

Let $R=\oplus_{g\in G} R_g$ and $R'=\oplus_{g\in G} R'_g$ be two $G$-graded algebras. The map $f\colon R\rightarrow R'$ is a homomorphism of $G$-graded algebras if $f$ is a homomorphism of algebras such that $f(R_g)\subseteq R'_g$ for every $g\in G$. Similarly one defines endomorphism, automorphism and isomorphism of $G$-graded algebras. Note that $R$ and $R'$ are isomorphic as $G$-graded algebras whenever there is an isomorphism of algebras from $R$ to $R'$ which respects the gradings.

The free $G$-graded associative algebra $\mathbb{F}\langle X \rangle_G$ over $\mathbb{F}$ is the free associative algebra freely generated by the set of variables $X=\cup_{g\in G} X_g$, where $X_g=\{x_{i,g}|i\geq 1\}$ are countable pairwise disjoint sets, with the $G$-grading in which the variables in $X_g$ are homogeneous of degree $g$. A polynomial $f\in\mathbb{F}\langle X\rangle_G$ is a $G$-graded polynomial identity for a $G$-graded associative algebra $R$ if vanishes under all graded substitutions, i.e., for any substitutions of the variables $x_{i,g}$ by elements of the homogeneous component $R_g$. We denote by $Id_{G}(R)$ the subset of $\mathbb{F}\langle X\rangle_G$ formed by all $G$-graded polynomial identities of $R$. It is well known that $Id_G(R)$ is a $G$-graded ideal of $\mathbb{F}\langle X\rangle_G$, which is invariant under all $G$-graded endomorphisms of $\mathbb{F}\langle X\rangle _G$; so it is called a $T_G$-ideal.
Clearly graded identities for $A$ are ``easier" to study than the ordinary ones. The former also provide useful information about the latter. Here we recall that if two $G$-graded algebras satisfy the same graded identities then they satisfy the same ordinary ones. Clearly the converse fails.

\section{Preliminaries on Crossed product}\label{prelimcrossed}

In this section, $\mathbb{F}$ is a finite field. We present the construction of crossed product and prove that, since the base field $\mathbb{F}$ is finite, every crossed product is isomorphic to a matrix algebra over $\mathbb{F}$. In the next section we use this fact in order to construct division gradings with cyclic support on $M_m(\mathbb{F})$, and to produce a classification of these gradings.

Let $\mathbb{E}/\mathbb{F}$ be a finite field extension  with Galois
group $G=\mbox{Gal} (\mathbb{E}/\mathbb{F})$. We form a vector space $A$ over $\mathbb{E}$ with basis $\{y_s \mid s\in G\}$, and so its elements have the form
\[\sum_{s\in G}\rho_s y_s, \ \rho_s\in \mathbb{E}.\]
Given a map 
\begin{equation}\label{definicaode2cociclo1}
\begin{array}{cccc}
k:&G\times G& \to &\mathbb{E}^* \\
&(s,t) & \mapsto & k_{s,t},
\end{array}
\end{equation}
define a product on $A$ as follows:
\[\left(\sum_{s\in G}\rho_s y_s\right)\left(\sum_{t\in G}\delta_t y_t\right)=\sum_{s\in G}\sum_{t\in G}k_{s,t}\rho_ss(\delta_t)y_{st}, \ \rho_s, \delta_t\in \mathbb{E}.\]
Equipped with  this product, $A$ becomes an algebra over $\mathbb{F}$,   
\[
y_sy_t=k_{s,t}y_{st}.
\]
Moreover $A$ is an associative algebra if and only if
\begin{equation}\label{definicaode2cociclo2}
    k_{s,t}k_{st,v}=s(k_{t,v})k_{s,tv}
    \end{equation}
for every $s$, $t$, $v\in G$. If (\ref{definicaode2cociclo2}) is satisfied, the map (\ref{definicaode2cociclo1}) is called a 2-cocycle (or a factor set) of $G$ in $\mathbb{E}^*$. From now on we assume (\ref{definicaode2cociclo2}), we call $A$ the crossed product of $\mathbb{E}$ and $G$ with respect to the 2-cocycle $k$,
and we write $A = (\mathbb{E}, G, k)$. 

\begin{theorem}\cite[Theorem 8.7]{J}. \label{teoremacentralsimples1} The crossed product $A=(\mathbb{E}, G, k)$ is a central simple algebra over $\mathbb{F}$. Moreover, $[A:\mathbb{F}]=m^2$ where $[\mathbb{E}:\mathbb{F}]=m$. 
\end{theorem}

\begin{corollary} \label{corolariocentralsimples2} If $[\mathbb{E}:\mathbb{F}]=m$, then $(\mathbb{E}, G, k)\cong M_m(\mathbb{F})$ as algebras over $\mathbb{F}$. 
\end{corollary}

\begin{proof}
We will use the notations from Theorem \ref{teoremacentralsimples1}. By Wedderburn's theorem, there exist $n\geq 1$ and a division algebra $\mathbb{D}$ over $\mathbb{F}$ such that $A \cong M_n(\mathbb{D})$, see \cite[Theorem 2.61]{B}.
Since $[A:\mathbb{F}]<\infty$ and $\mathbb{F}$ is finite it follows that $A$ is finite and so $\mathbb{D}$ is also finite. Since $\mathbb{D}$ is a finite division ring, Wedderburn's little theorem implies that $\mathbb{D}$ is a field, 
see \cite[Theorem 8.23]{R}. Therefore the centre of $M_n(\mathbb{D})$ is $\mathbb{D}$,  $\mathbb{D} \cong \mathbb{F}$ by Theorem \ref{teoremacentralsimples1}, and $A \cong M_n(\mathbb{F})$. Since $n^2=[A:\mathbb{F}]=m^2$, we obtain $n=m$. 
\end{proof} 
 
\section{Division Gradings on finite Matrix Algebras with Cyclic Support}\label{divisiongradings}

In this section, $\mathbb{F}$ is a finite field with $q$ elements, $G$ is a group with identity element $e$, $T$ is a finite cyclic subgroup of $G$ and $m\geq 1$. We will classify the division $G$-gradings on $M_m(\mathbb{F})$ with support $T$. 

For the particular case $|T|=m$, we will first construct a division $G$-grading on $M_m(\mathbb{F})$ with support $T$ as follows:
\begin{itemize}
\item[(a)]Let $\mathbb{E}/\mathbb{F}$ be a finite field extension  with $[\mathbb{E}:\mathbb{F}]=m$. 
\item[(b)]Since $\mathbb{E}/\mathbb{F}$ is a Galois extension,  $|\mbox{Gal} (\mathbb{E}/\mathbb{F})|=[\mathbb{E}:\mathbb{F}]=m$. 
\item[(c)] If $p$ is the characteristic of $\mathbb{F}$, \cite[Theorem 4.12]{R} implies that Gal$(\mathbb{E}/\mathbb{Z}_p)$ is cyclic and so its subgroup Gal$(\mathbb{E}/\mathbb{F})$ is also cyclic.  
\item[(d)] Since $T$ and Gal$(\mathbb{E}/\mathbb{F})$ are cyclic groups of order $m$, they are isomorphic. Hence we can identify them, see below.
\item[(e)] Let $\psi\colon T\to \mathrm{Gal}(\mathbb{E}/\mathbb{F})$ be a 
group isomorphism, and denote $\psi(u)=\psi_u$.
\item[(f)] Denote by $D_{T, \mathbb{E},\psi}$ the vector space over $\mathbb{E}$ with basis $\{Y_s \mid s\in T\}$. Its elements are of the form
\[
\sum_{s\in T}\rho_s Y_s, \quad \rho_s\in \mathbb{E}.
\]
For the 2-cocycle $k$ of  $\mbox{Gal} (\mathbb{E}/\mathbb{F})$ in $\mathbb{E^*}$ given by $k_{s,t}=1$ for every $s$, $t \in \mbox{Gal} (\mathbb{E}/\mathbb{F})$, there exists an $\mathbb{E}$-linear isomorphism $D_{T, \mathbb{E},\psi} \to A=(\mathbb{E}, \mbox{Gal} (\mathbb{E}/\mathbb{F}), k)$ such that $Y_s \mapsto  y_{\psi_s}$. Thus $A$ induces on $D_{T, \mathbb{E},\psi}$ the structure of an  algebra over $\mathbb{F}$ with multiplication rule given by
\[
\left(\sum_{s\in T}\rho_s Y_s\right)\left(\sum_{t\in T}\delta_t Y_t\right)=\sum_{s\in T}\sum_{t\in T}\rho_s\psi_s(\delta_t)Y_{st}, \quad \rho_s, \delta_t\in \mathbb{E}.
\]
Note that there exists only one division $G$-grading on $D_{T, \mathbb{E},\psi}$ whose support equals $T$, and such that 
\[
\deg (\rho Y_s)=s \quad \mbox{for every} \quad  s\in T, \quad \rho \in \mathbb{E}.
\]
\item[(g)] By Corollary \ref{corolariocentralsimples2}, there exists an isomorphism $\varphi\colon D_{T, \mathbb{E},\psi} \to M_m(\mathbb{F})$ of algebras over $\mathbb{F}$, and so $\varphi$ induces a division $G$-grading on $M_m(\mathbb{F})$ with support $T$, as desired. Note that  
\[ 
\deg (\varphi(\rho Y_s))=s \quad \mbox{for every} \quad  s\in T, \quad \rho \in \mathbb{E}.
\]
\end{itemize}

Before stating one of the main results of this paper, we would like to emphasize the relevance of the recent paper \cite{BDW} and its connection with the theory developed in this work. The pertinent remarks will be given at the end of this section.

\begin{theorem}\label{DivisionMatrix}
	Let $D=M_m(\mathbb{F})$ be a $G$-graded division algebra, over a finite field $\mathbb{F}$, with support $T$. If $T$ is a cyclic group and $\mathbb{E}=D_e$ then:
 \begin{enumerate}
     \item[(a)] $\mathbb{E}$ is a finite field.
     \item[(b)] $|T|=[\mathbb{E}:\mathbb{F}]=m$.
     \item[(c)] There exists an isomorphism $\psi:T\rightarrow \mathrm{Gal} ( \mathbb{E}/\mathbb{F})$ of groups such that $D\cong D_{T, \mathbb{E},\psi}$ as $G$-graded algebras over $\mathbb{F}$.
 \end{enumerate}
\end{theorem}

    \begin{proof}
Since $\mathbb{E}$ is a finite division ring, Wedderburn's little theorem implies that $\mathbb{E}$ is a field.

Let $\textbf{t}$ be a generator of $T$ and denote $n=|T|$. Consider a nonzero homogeneous element $X_{\textbf{t}}\in D$ of degree $\textbf{t}$, and denote $X_{\textbf{t}^i}:=X_{\textbf{t}}^i$ for $0\leq i <n$.  Note that $D_{u}=\mathbb{E}X_u$ for every $u\in T$ and, in particular, $[D:\mathbb{E}]=n$. For each $u\in T$ the map $\psi_u\colon \mathbb{E} \to \mathbb{E}$ given by 
\[
\psi_u(\delta)= X_u\delta X_u^{-1}, \quad \delta \in \mathbb{E},
\] 
lies in $\mathrm{Gal}\, (\mathbb{E}/\mathbb{F})$. The map $\psi\colon T\to \mathrm{Gal}(\mathbb{E}/\mathbb{F})$ given by $\psi(u)=\psi_u$ is a homomorphism of groups. Indeed let $u$, $v\in T$, since $X_uX_v\in D_{uv}=\mathbb{E}X_{uv}$ there exists a non-zero element $k_{u,v}\in \mathbb{E}$ such that $X_uX_v=k_{u,v}X_{uv}$. For $\delta \in \mathbb{E}$ we have
\[
(\psi_u\psi_v)(\delta)=X_uX_v\delta(X_uX_v)^{-1}=k_{u,v}X_{uv}\delta X_{uv}^{-1}k_{u,v}^{-1}\stackrel{*}{=}X_{uv}\delta X_{uv}^{-1}=\psi_{uv}(\delta),
\]
the equality $(*)$ above holds because $\mathbb{E}$ is commutative. Hence we conclude that $\psi_{uv}=\psi_u\psi_v$ for all $u$, $v\in T$ and as a consequence the map $\psi$ is a homomorphism of groups.

\smallskip

\noindent \emph{Claim 1.} $\psi$ is an isomorphism of groups.

\noindent \emph{Proof of Claim 1.} For every 
$\delta\in \mathbb{E}$ and $u$, $v\in T$ we have
\begin{equation}\label{igualdade6ja}
    X_u\delta=\psi_u(\delta)X_u \ \ \mbox{and} \ \ X_uX_v=X_vX_u.
    \end{equation}
If $u\in \mathrm{Ker}\, \psi$ then $X_u\delta=\delta X_u$ and  
\[
X_u(\delta  X_v)=X_u\delta X_v=\delta X_u X_v=\delta X_vX_u=(\delta X_v)X_u
\] 
for every $\delta\in \mathbb{E}$ and $v\in T$, that is $X_u$ lies in the centre of $D$ which is $\mathbb{F}\subseteq D_e$. Therefore, if $u\in \mathrm{Ker}\, \psi$ then $u=e$ and $\psi$ is an injective homomorphism. The subfield of $\mathbb{E}$ fixed by the image of $\psi$ lies in the centre of $D$, hence it equals $\mathbb{F}$. Since $\mathbb{E}/\mathbb{F}$ is a Galois extension the image of $\psi$ is $\mathrm{Gal}(\mathbb{E}/\mathbb{F})$, therefore $\psi$ is an isomorphism of groups. We finish the proof of Claim 1.

\smallskip

\noindent \emph{Claim 2.} $|T|=[\mathbb{E}:\mathbb{F}]=m$.

\noindent \emph{Proof of Claim 2.} By Claim 1, we obtain $n=|T|  = | \mathrm{Gal}(\mathbb{E}/\mathbb{F})|=[\mathbb{E}:\mathbb{F}]$. Therefore,
\[
m^2=[D:\mathbb{F}]=[D:\mathbb{E}][\mathbb{E}: \mathbb{F}]=n^2,
\]
and hence $n=m$. We finish the proof of Claim 2.

Denote $\alpha=X_{\textbf{t}}^m$ and note that $\alpha \in \mathbb{E}$. Since \[\psi_{\textbf{t}}(\alpha)X_{\textbf{t}}=X_{\textbf{t}}\alpha=X_{\textbf{t}}^{m+1}=\alpha X_{\textbf{t}}\] 
we conclude that $\psi_{\textbf{t}}(\alpha)=\alpha$, hence $\psi_u(\alpha)=\alpha$ for every $u\in T$. Then $\alpha$ is fixed by every element of $\mathrm{Gal}( \mathbb{E}/\mathbb{F})$, and this implies that $\alpha\in \mathbb{F}$.   Denote by $N\colon\mathbb{E^*} \to \mathbb{F^*}$ the homomorphism of groups given by
\[
N(\delta)=\prod_{\sigma \in \mathrm{Gal}( \mathbb{E}/\mathbb{F})}\sigma(\delta).\]
We say that $N(\delta)$ is the norm of $\delta$ and, by \cite[Page 838]{DF}, the map $N$ is surjective. Hence there exists $\overline{\delta}\in \mathbb{E}^*$ such that $N(\overline{\delta})=\alpha^{-1}$. We have 
\[
(\overline{\delta} X_{\textbf{t}})^m=\left(\prod_{i=0}^{m-1}\psi_{{\textbf{t}}^i}(\overline{\delta})\right)X_{\textbf{t}}^m=N(\overline{\delta})\alpha=1.
\] 
Thus we replace $X_{\textbf{t}}$ with $\overline{\delta}X_{\textbf{t}}$ if necessary and assume, without loss of generality, that $X_{\textbf{t}}^m=1$. Then $X_uX_v=X_{uv}$ for every $u$, $v\in T$ and, by (\ref{igualdade6ja}), it follows that 
\[
\left(\sum_{u\in T}\rho_u X_u\right)\left(\sum_{v\in T}\delta_v X_v\right)=\sum_{u\in T}\sum_{v\in T}\rho_u\psi_u(\delta_v)X_{uv}, \quad \rho_u, \delta_v\in \mathbb{E}.
\]
Hence the map $D \to D_{T, \mathbb{E},\psi}$ given by
\[\sum_{u\in T}\rho_uX_u\mapsto \sum_{u\in T}\rho_uY_u\] is an isomorphism of $G$-graded algebras over $\mathbb{F}$.
\end{proof}

\begin{theorem}\label{divisioniso}
Let $\varphi\colon T\to \mathrm{Gal}(\mathbb{E}/\mathbb{F})$ and $\psi\colon T\to \mathrm{Gal}(\mathbb{E}/\mathbb{F})$ be isomorphisms of groups, where $\mathbb{E}/\mathbb{F}$ is a finite field extension. Then $\varphi = \psi$ if and only if $D_{T, \mathbb{E},\varphi} \cong D_{T, \mathbb{E},\psi}$ as $G$-graded algebras over $\mathbb{F}$. 
\end{theorem}

\begin{proof}
Suppose $D_{T, \mathbb{E},\varphi} \cong D_{T, \mathbb{E},\psi}$ and let $f\colon  D_{T, \mathbb{E},\varphi} \to D_{T, \mathbb{E},\psi}$ be an isomorphism of $G$-graded algebras over $\mathbb{F}$. Then 
\[
f((D_{T, \mathbb{E},\varphi})_e) =(D_{T, \mathbb{E},\psi})_e,
\]
and the restriction map $f\colon (D_{T, \mathbb{E},\varphi})_e \to (D_{T, \mathbb{E},\psi})_e$ is an isomorphism of $\mathbb{F}$-algebras. 
Consider the following isomorphisms of $\mathbb{F}$-algebras
\[\begin{array}{cccc}
i \colon &  \mathbb{E} & \to & (D_{T, \mathbb{E},\varphi})_e\\
& \delta & \mapsto & \delta X_{e}
\end{array} \ \ \mbox{and} \ \ \ 
\begin{array}{cccc}
j \colon &  \mathbb{E} & \to & (D_{T, \mathbb{E},\psi})_e\\
& \delta & \mapsto & \delta Y_{e}.
\end{array}
\]
Here $\{X_u \mid u\in T\}$ and $\{Y_u \mid u\in T\}$ are linear bases for  $D_{T, \mathbb{E},\varphi} $ and $D_{T, \mathbb{E},\psi}$ over $\mathbb{E}$, respectively, and $X_u$, $Y_u$ are homogeneous of degree $u$, moreover $X_e$, $Y_e$ are the units of $D_{T, \mathbb{E},\varphi}$ and $D_{T, \mathbb{E},\psi}$, respectively. 
If $g=j^{-1}\circ f\circ i$ then 
\[
g \in \mathrm{Gal}( \mathbb{E}/\mathbb{F}) \ \ \ \mbox{and} \ \ \ f(\delta X_e)=g(\delta)Y_{e}
\] 
for every $\delta \in \mathbb{E}$. In fact,
\[
f(\delta X_e)=f(i(\delta))=j(g(\delta))=g(\delta)Y_{e}.
\]

For every $u\in T$ we shall prove that $\varphi (u)=\psi (u)$, that is, $\varphi_u=\psi_u$.
If $\delta\in\mathbb{E}$ then 
\begin{equation} \label{igualdadejulho14}f(X_u (\delta X_e))=f(\varphi_u(\delta) X_u)=f(\varphi_u(\delta)X_e)f(X_u)=g(\varphi_u(\delta))f(X_u).
\end{equation}
Let $\delta_u \in \mathbb{E}$ be such that $f(X_u)=\delta_uY_u$. By $(\ref{igualdadejulho14})$, and since $f$ is an isomorphism, it follows that 
\begin{align*}
 g(\varphi_u(\delta))f(X_u) =&f(X_u (\delta X_e))=
 f(X_u)f(\delta X_{e})\\
 =& (\delta_u Y_u)(g(\delta)Y_{e})=\delta_u \psi_u(g(\delta)) Y_u= \psi_u(g(\delta))\delta_u Y_u  \\
 =& \psi_u(g(\delta)) f(X_u),
\end{align*}
and so $g(\varphi_u(\delta))=\psi_u(g(\delta))$ for every $\delta\in\mathbb{E}$. This implies $g\circ \varphi_u=\psi_u\circ g$, hence $\varphi_u= g^{-1}\circ \psi_u\circ g$. Since $\mathrm{Gal}(\mathbb{E}/\mathbb{F})$ is an abelian group, we obtain $\varphi_u=\psi_u$, as desired. \end{proof}

 \begin{notation}
Henceforth we fix an extension $\mathbb{E}_m$ of the field $\mathbb{F}$ such that $[\mathbb{E}_m:\mathbb{F}]=m$ and a generator $\emph{\textbf{t}}$ of the cyclic subgroup $T$ of $G$.
 \end{notation}

\noindent {\bf Comments.} Let $a\in \mathbb{Z}$ be such that $|\mathbb{F}|=p^a$ where  $\mathrm{char}\, \mathbb{F}=p$.
Since $[\mathbb{E}_m:\mathbb{F}]=m$, we have 
$|\mathbb{E}_m|=p^{ma}$.  Let $\zeta\colon\mathbb{E}_m \to \mathbb{E}_m$ be the Frobenius map given by
\[
\zeta(x)=x^p, \  x\in \mathbb{E}_m.
\] 
Since $\zeta$ is a generator of the cyclic group $\mbox{Gal}(\mathbb{E}_m / \mathbb{Z}_p)$, $|\mbox{Gal}(\mathbb{E}_m / \mathbb{Z}_p)|=[\mathbb{E}_m:\mathbb{Z}_p]=ma$, $|\mbox{Gal}(\mathbb{E}_m / \mathbb{F})|=m$ and $\mbox{Gal}(\mathbb{E}_m / \mathbb{F})$ is a subgroup of $\mbox{Gal}(\mathbb{E}_m / \mathbb{Z}_p)$, it follows that
$\mbox{Gal}(\mathbb{E}_m / \mathbb{F})$ is generated by $\zeta^a$. Since every generator of $\mbox{Gal}(\mathbb{E}_m / \mathbb{F})$ has the form $(\zeta^a)^b=\zeta^{ab}$ such that gcd$(m,b)=1$   
and $1 \leq b \leq m-1$, there exists an isomorphism of groups
\begin{equation}\label{11dejaneiro}
\psi^{(m,b)}\colon T\to \mathrm{Gal}(\mathbb{E}_m/\mathbb{F}) \ \ \mbox{such that} \ \ \psi^{(m,b)}(\textbf{t})=\zeta^{ab}.
\end{equation}
Clearly $\psi^{(m,b)}$ depends on the choice of the generator $\textbf{t}$ of $T$. Note that there exist $\phi(m)$ such isomorphisms where $\phi$ is the Euler's totient function. Moreover, if $\psi\colon T\to \mathrm{Gal}(\mathbb{E}_m/\mathbb{F})$ is an isomorphism of groups, then $\psi=\psi^{(m,b)}$ for some $b$.

\smallskip

We shall use the above notation $\psi^{(m,b)}$ in the next result.

\begin{theorem} \label{teorema2de16denovembro}
     If $\mathbb{K}/\mathbb{F}$ is a field extension, $[\mathbb{K}:\mathbb{F}]=|T|=m$ and
      $\varphi\colon T\to \mathrm{Gal} ( \mathbb{K}/\mathbb{F})$ is an isomorphism of groups, then there exists an integer $b$ such that  
     \[\emph{gcd}(m,b)=1, \ 1 \leq b \leq m-1 \ \  and \ \ D_{T, \mathbb{K},\varphi}\cong 
     D_{T, \mathbb{E}_m,\psi^{(m,b)}}\] as $G$-graded algebras over $\mathbb{F}$. Moreover, there are exactly $\phi(m)$ division $G$-gradings on $M_m(\mathbb{F})$ with support $T$, that are not pairwise isomorphic.
\end{theorem}

\begin{proof}
Since $\mathbb{E}_m$ and $\mathbb{K}$ are splitting fields of the polynomial
\[
w(x)=x^{q^m}-x
\]
over $\mathbb{F}$, where $|\mathbb{E}_m|=|\mathbb{K}|=q^m$, by \cite[Theorem 3.131]{R} there exists an isomorphism of fields $\eta\colon \mathbb{K}\to  \mathbb{E}_m$ that fixes $\mathbb{F}$ pointwise. 
The map $\overline{\eta}\colon \mathrm{Gal} ( \mathbb{K}/\mathbb{F}) \to \mathrm{Gal} ( \mathbb{E}_m/\mathbb{F})$ given by
\[\overline{\eta}(f)=\eta\circ f \circ \eta^{-1}\]
is an isomorphism of groups and 
 $\overline{\eta}\circ \varphi\colon T\to \mathrm{Gal} ( \mathbb{E}_m/\mathbb{F})$ is one as well.

\

\noindent \emph{Claim 1.} $D_{T, \mathbb{K},\varphi}\cong 
     D_{T, \mathbb{E}_m,\overline{\eta}\circ \varphi}$ as $G$-graded algebras over $\mathbb{F}$.

\noindent \emph{Proof of Claim 1.} Let $\{X_s \ | \ s\in T\}$ and $\{Y_s \ | \ s\in T\}$ be bases for the vector spaces 
$D_{T, \mathbb{K},\varphi}$ and $ 
     D_{T, \mathbb{E}_m,\overline{\eta}\circ \varphi}$ over 
$\mathbb{K}$ and $\mathbb{E}_m$, respectively. Here $X_s$ and $Y_s$ are homogeneous of degree $s$. 
Define $\theta\colon D_{T, \mathbb{K},\varphi}\to 
     D_{T, \mathbb{E}_m,\overline{\eta}\circ \varphi}$ by 
\[
\theta \left(\sum_{s\in T}k_s X_s\right)=
\sum_{s\in T}\eta(k_s) Y_s, \ \ k_s \in \mathbb{K}.
\]
The map $\theta$ is a linear isomorphism of $G$-graded vector spaces over $\mathbb{F}$. For the proof that $\theta$ is an isomorphism of algebras over $\mathbb{F}$, it is enough to prove that
\[
\theta((k_sX_s)(k_uX_u))=\theta(k_sX_s)\theta(k_uX_u),
\]
for every $s$, $u\in T$ and $k_s$, $k_u\in \mathbb{K}$.
We have
\begin{align*}
\theta(k_sX_s)\theta(k_uX_u)=& (\eta(k_s)Y_s)(\eta(k_u)Y_u)=
\eta(k_s)(\overline{\eta}\circ \varphi)_s(\eta(k_u))Y_{su}\\
=&
\eta(k_s)\overline{\eta}( \varphi_s)(\eta(k_u))Y_{su} \\
=&
\eta(k_s)( \eta \circ \varphi_s \circ \eta^{-1})(\eta(k_u))Y_{su}
=\eta(k_s) \eta ( \varphi_s (k_u))Y_{su} \\
=&\eta(k_s \varphi_s (k_u))Y_{su}=\theta (  
k_s \varphi_s (k_u)X_{su} ) \\
=&\theta((k_sX_s)(k_uX_u))
\end{align*}
as desired.

\

\noindent \emph{Claim 2.} There exists an integer $b$ such that 
     \[\mbox{gcd}(m,b)=1, \ 1 \leq b \leq m-1 \ \  \mbox{and} \ \ D_{T, \mathbb{K},\varphi}\cong 
     D_{T, \mathbb{E}_m,\psi^{(m,b)}}\] as $G$-graded algebras over $\mathbb{F}$.

\noindent \emph{Proof of Claim 2.} By the comments before this theorem, there exists an integer $b$ such that $D_{T, \mathbb{E}_m,\overline{\eta} \circ \varphi}= D_{T, \mathbb{E}_m,\psi^{(m,b)}}$ as $G$-graded algebras over $\mathbb{F}$. Now we use Claim 1.

\

By Theorem \ref{DivisionMatrix}, Claim 2, and Theorem \ref{divisioniso} there are exactly $\phi(m)$ division $G$-gradings on $M_m(\mathbb{F})$ with support $T$ which are not pairwise isomorphic.
 \end{proof}

In what follows, we make some comments concerning \cite{BDW}. In it, the authors study cyclic algebras over an arbitrary field $F$. Since the Galois groups considered here are abelian, it follows that Theorems \ref{divisioniso} and \ref{teorema2de16denovembro} are particular cases of \cite[Proposition 2.1, Corollary 2.2, Corollary 2.3]{BDW}. These results were proved independently by the authors of both papers, and we have decided to include our proofs for the sake of completeness and because of the ideas that will be used in the subsequent sections.

Another point worth mentioning concerns Theorem \ref{DivisionMatrix} and \cite[Theorem 5.1]{BDW}. Although these are distinct results, they overlap when $m$ is a prime number. The paper \cite{BDW} deals with the case $m$ prime. We remark that finite graded division rings, graded by an abelian group, were classified up to isomorphism in \cite[Theorem 5.14]{EK0}.

\section{One example of division grading on matrix algebras}\label{example}

In this section we use some ideas from \cite{KSD} in order to construct a divison $G$-grading on $M_2(\mathbb{F})$, and to make the next section more transparent and easier to understand.

Let $\mathbb{F}$ be the field with $2$ elements, $\mathbb{F}=\{0,1\}$, and let $G$ be a group with some subgroup $T=\{e,a\}$ of order $2$. Consider $\mathbb{E}$, the splitting field of the polynomial $f(x)=x^4-x$ 
over $\mathbb{F}$, then 
\[
\mathbb{E}=\{0,1, \alpha, \alpha^2\},
\] 
where $\alpha^2+\alpha=1$. Note that  the field $\mathbb{E}$ is a finite extension of $\mathbb{F}$ and $[\mathbb{E}:\mathbb{F}]=2$. Moreover, $\mathbb{E}$ is a Galois extension of $\mathbb{F}$ and $|\mbox{Gal} (\mathbb{E}/\mathbb{F})|=[\mathbb{E}:\mathbb{F}]=2$. Hence, $\mbox{Gal} (\mathbb{E}/\mathbb{F})$ is isomorphic to the cyclic group $T$. Observe that 
\[
\mbox{Gal} (\mathbb{E}/\mathbb{F})=\{\psi_e, \psi_a\},
\] 
where $\psi_e\colon \mathbb{E}\to \mathbb{E}$ is the identity map and $\psi_a\colon  \mathbb{E}\to \mathbb{E}$ is given by 
\[
\psi_a (0)=0, \  \psi_a(1)=1, \  \psi_a(\alpha)=\alpha^2 ,\  \mbox{and} \ \psi_a(\alpha^2)=\alpha.
\] 
Let $\psi\colon T\to \mbox{Gal} (\mathbb{E}/\mathbb{F})$ be the map defined by $\psi(e)=\psi_e$ and $\psi(a)=\psi_a$. It is clear that $\psi$ is a group isomorphism.

Consider $D_{T, \mathbb{E},\psi}$ the vector space over $\mathbb{E}$ with basis $\{Y_e, Y_a\}$. Its elements have the form
\[
\rho_e Y_e + \rho_a Y_a,  \quad  \rho_e,\rho_a\in \mathbb{E}.
\]
For the 2-cocycle $k$ of  $\mbox{Gal} (\mathbb{E}/\mathbb{F})$ in $\mathbb{E^*}$ given by $k_{s,t}=1$ for every $s$, $t \in \mbox{Gal} (\mathbb{E}/\mathbb{F})$, there exists an $\mathbb{E}$-linear isomorphism $D_{T, \mathbb{E},\psi} \to A=(\mathbb{E}, \mbox{Gal} (\mathbb{E}/\mathbb{F}), k)$ such that $Y_s \mapsto y_{\psi_s}$. Thus $A$ induces on $D_{T, \mathbb{E},\psi}$ a structure of an algebra over $\mathbb{F}$ with multiplication given by the following rule:
\[
\left(\rho_e Y_e + \rho_a Y_a\right)\left(\delta_e Y_e + \delta_a Y_a\right)=\left(\rho_e\delta_e+\rho_a\psi_a(\delta_a)\right) Y_e + \left(\rho_e \delta_a + \rho_a\psi_a(\delta_e)\right)Y_a.
\]
We have the following $G$-grading on $D_{T, \mathbb{E},\psi}$ whose support equals $T$: 
\begin{align*}
  (D_{T, \mathbb{E},\psi})_e=& \{0,Y_e,\alpha Y_e, \alpha^2 Y_e\},\\
(D_{T, \mathbb{E},\psi})_a=&\{0, Y_a,\alpha Y_a, \alpha^2 Y_a\},\\
(D_{T, \mathbb{E},\psi})_g=&\{0\} \ \mbox{for all} \ g\notin T,
\end{align*}
and the next table describes the products of the nonzero homogeneous elements in $D_{T, \mathbb{E},\psi}$:
\begin{center}
\begin{tabular}{|r||c|c|c|c|c|c|}
\hline
& $Y_e$ & $\alpha Y_e$  & $\alpha^2 Y_e$ &$Y_a$ & $\alpha Y_a$& $\alpha^2 Y_a$\\
\hline\hline
$Y_e$ &  $Y_e$ & $\alpha Y_e$  & $\alpha^2 Y_e$ &$Y_a$ & $\alpha Y_a$& $\alpha^2 Y_a$ \\
\hline
$\alpha Y_e$ & $\alpha Y_e$ & $\alpha^2 Y_e$ & $Y_e$ & $\alpha Y_a$ & $\alpha^2 Y_a$ & $Y_a$\\
\hline  
$\alpha^2 Y_e$ &$\alpha^2 Y_e$ & $Y_e$ & $\alpha Y_e$ & $\alpha^2 Y_a$ & $Y_a$& $\alpha Y_a$ \\
\hline
$Y_a$ &$Y_a$ & $\alpha^2 Y_a$ & $\alpha Y_a$ & $Y_e$ & $\alpha^2 Y_e$ & $\alpha Y_e$ \\
\hline
$\alpha Y_a$ &$\alpha Y_a$ & $Y_a$ & $\alpha^2 Y_a$ & $\alpha Y_e$ &$Y_e$  & $\alpha^2 Y_e$ \\
\hline
$\alpha^2 Y_a$ &$\alpha^2 Y_a$ & $\alpha Y_a$ & $Y_a$ & $\alpha^2 Y_e$ & $\alpha Y_e$ & $ Y_e$ \\
\hline
\end{tabular}
\end{center}
Hence, according to the table above,  $D_{T, \mathbb{E},\psi}$ is a 
$G$-graded division algebra with the previous $G$-grading.

Let $f\colon D_{T,\mathbb{E},\psi}\to M_2(\mathbb{F})$ be the $\mathbb{F}$-linear map such that
\[
f(Y_e+ \alpha Y_a)=e_{1,2},\,  f(Y_e+\alpha^2 Y_a)=e_{2,1},\, f(\alpha Y_e+Y_a)=e_{1,1},\,  f(\alpha^2 Y_e+ Y_a)=e_{2,2}.
\]
We have that $f$ is an isomorphism of algebras over $\mathbb{F}$, and $f$ induces the following division $G$-grading on $M_2(\mathbb{F})$ with support $T$:
\begin{align*}
    M_2(\mathbb{F})_e=&\left\lbrace
\begin{pmatrix}
0 & 0 \\
0 & 0
\end{pmatrix},
\begin{pmatrix}
1 & 0 \\
0 & 1
\end{pmatrix},
\begin{pmatrix}
1 & 1 \\
1 & 0
\end{pmatrix},\begin{pmatrix} 0 & 1 \\
1 & 1
\end{pmatrix}\right\rbrace ,\\
M_2(\mathbb{F})_a=& \left\lbrace
\begin{pmatrix}
0 & 0 \\
0 & 0
\end{pmatrix},
\begin{pmatrix}
0 & 1 \\
1 & 0
\end{pmatrix},
\begin{pmatrix} 1 & 1 \\
0 & 1
\end{pmatrix}, \begin{pmatrix} 1 & 0 \\
1 & 1
\end{pmatrix}\right\rbrace ,\\
M_2(\mathbb{F})_g=&\left\lbrace \begin{pmatrix}
0 & 0 \\
0 & 0
\end{pmatrix}\right\rbrace \ \mbox{for every} \ g\notin T.
\end{align*}
As a consequence, $M_2(\mathbb{F})$, endowed with this grading, is isomorphic to $D_{T,\mathbb{E},\psi}$, as $G$-graded algebras over $\mathbb{F}$.

\section{Gradings on Finite Matrix Algebras}\label{gradingsmatrix}

 In this section, we recall that $\mathbb{F}$ is a finite field with $q$ elements, $G$ is a group such that every finite subgroup of $G$ is cyclic, and $n\geq 1$ is a positive integer. Our goal here is to classify the $G$-gradings on $M_n(\mathbb{F})$.

The classification of the $G$-gradings on finite matrix algebras is a 
consequence of the classification of the division $G$-gradings on these algebras obtained in Section \ref{divisiongradings} in a standard way, see \cite{BZ} and \cite[Section 2.3]{EK}. In order to proceed we develop the notion of induced grading on the tensor products introduced in \cite{BZ}.

Let $D=\oplus_{g\in G}D_g$ be a $G$-graded algebra 
and let $R=M_t(\mathbb{K})\otimes_{\mathbb{K}}D$, here the field $\mathbb{K}$ is arbitrary, not necessarily finite. Given ${\bf g}=(g_1,\dots, g_t)\in G^t$ let $R_g=\mathrm{Span}\{e_{i,j}\otimes d \mid d\in D_h, g_ihg_j^{-1}=g\}$, where $e_{i,j}$ are the usual matrix units. Then $R=\oplus_{g\in G}R_g$ and this decomposition is a $G$-grading on $R$. We refer to this grading as the grading on the tensor product induced by ${\bf g}$. Note that in the induced grading we have 
\begin{equation} \label{equacao2de5deoutubro}
\deg(e_{i,j}\otimes d)= g_i (\deg d) g_j^{-1},
\end{equation}
where $d$ is a homogeneous element of $D$. 

A classification of the graded simple algebras is given in \cite[Theorem 2.6]{EK}, as a direct consequence of this result and the comments following it we obtain the following theorem.
 
\begin{theorem} \label{teorema28deagosto}
Let $\mathbb{K}$ be a field, let $\tilde{G}$ be a group and let $R$ be a finite dimensional $\tilde{G}$-graded algebra over $\mathbb{K}$. If $R$ is a graded simple algebra then there exist a $\tilde{G}$-graded division algebra $D$ over $\mathbb{K}$, a positive integer $t$ and a tuple ${\bf g}=(g_1,\dots, g_t)\in \tilde{G}^t$ such that $R$ is isomorphic as a $\tilde{G}$-graded algebra to $M_t(\mathbb{K})\otimes_{\mathbb{K}} D$ with the grading induced by the tuple ${\bf g}$.
\end{theorem}

The above theorem reduces the classification of the $\tilde{G}$-gradings on matrix algebras to the classification of the division $\tilde{G}$-gradings on these algebras. We draw the readers' attention to the fact that in the theorem above $\tilde{G}$ is an arbitrary group and $\mathbb{K}$ is an arbitrary field. Hence as a consequence of Theorem \ref{DivisionMatrix} we obtain the following theorem.

\begin{theorem}\label{teorema16denovembro}
Let $\mathbb{F}$ be a finite field and let $G$ be a group such that every finite subgroup of $G$ is cyclic.
If $M_n(\mathbb{F})$ is a $G$-graded algebra, then there exist an integer  $t\geq 1$, a finite subgroup $T$ of $G$, a finite field extension $\mathbb{E}/\mathbb{F}$ of degree 
$|T|$ and 
an isomorphism $\psi\colon T\to \mathrm{Gal} ( \mathbb{E}/\mathbb{F})$ of groups such that $M_n(\mathbb{F})$ is isomorphic to $M_t(\mathbb{F})\otimes_{\mathbb{F}} D_{T,\mathbb{E},\psi}$ as $G$-graded algebras. 
\end{theorem} 

\begin{proof}
We recall that $M_n(\mathbb{F})$ is simple as an $\mathbb{F}$-algebra,  hence $M_n(\mathbb{F})$ is a $G$-graded simple algebra for any $G$-grading it might be equipped with. Thus by Theorem \ref{teorema28deagosto}, there exist $t\geq 1$, ${\bf g}=(g_1,\dots, g_t)\in G^t$, and a $G$-graded division algebra $D$ such that $M_n(\mathbb{F})\cong M_t(\mathbb{F})\otimes_{\mathbb{F}} D$ as $G$-graded algebras, where the latter algebra is equipped with the grading induced by ${\bf g}$. Observe that $D$ is a simple algebra because $M_n(\mathbb{F})$ is simple. 
    
By Wedderburn's theorem, there exist $m\geq 1$ and a division algebra $\mathbb{L}$ over $\mathbb{F}$ such that $D \cong M_m(\mathbb{L})$ as algebras over $\mathbb{F}$, see \cite[Theorem 2.61]{B}. Since $M_n(\mathbb{F})$ is finite, we have that $\mathbb{L}$ is also finite. Since $\mathbb{L}$ is a finite division ring, Wedderburn's little theorem implies that $\mathbb{L}$ is a field, see \cite[Theorem 8.23]{R}. Therefore
\[
M_n(\mathbb{F})\cong M_t(\mathbb{F})\otimes_{\mathbb{F}} M_m(\mathbb{L})\cong M_{tm}(\mathbb{L}),
\]
as algebras over the field $\mathbb{F}$.
As for the centres, one has the isomorphism 
\[
\mathbb{F} = Z(M_n(\mathbb{F}))\cong Z(M_{tm}(\mathbb{L})) = \mathbb{L},
\] 
and so $D\cong M_m(\mathbb{F})$. Now we apply Theorem \ref{DivisionMatrix}, this finishes the proof.
\end{proof}

In order to proceed we need to establish when two of the tensor products in Theorem \ref{teorema28deagosto} are isomorphic. Of course some
different tuples ${\bf g}\in \tilde{G}^t$ will determine different gradings on $M_t(\mathbb{K})\otimes_{\mathbb{K}} D$ while others might produce isomorphic gradings. 

It follows from the proof of \cite[Lemma 1.3]{AH} that if we multiply the entries of ${\bf g}$ by elements of $\mathrm{Supp}\, D$ and then permute the entries of ${\bf g}$ the resulting grading is isomorphic to the original one. Hence we may assume in Theorem \ref{teorema28deagosto} that 
\begin{equation}\label{tupleg}
{\bf g}=(\underbrace{g_1,\ldots ,g_1}_{k_1},\underbrace{g_2,\ldots ,g_2}_{k_2},\ldots ,\underbrace{g_s,\ldots ,g_s}_{k_s}),
\end{equation}
where
\begin{equation} \label{equacao5deoutubro}
    g_i^{-1}g_j\notin \mathrm{Supp}\, D \quad \mbox{for all} \quad  i\neq j.
\end{equation}

We remark that $t=k_1+\cdots +k_s$. The following notation will be useful from now on.

 \begin{notation}\label{notation22set}
We denote by $\mathcal{M}(D,\kappa,\gamma)$, where $\kappa=(k_1,\dots, k_s)$ and $\gamma=(g_1,\dots, g_s)$, the grading on $M_t(\mathbb{K})\otimes_{\mathbb{K}} D$ induced by the tuple (\ref{tupleg}).
 \end{notation}    

 The remarks above imply that in Theorem \ref{teorema28deagosto} we may replace $M_t(\mathbb{K})\otimes_{\mathbb{K}} D$ with $\mathcal{M}(D,\kappa,\gamma)$ for a suitable triple $(D,\kappa,\gamma)$. Hence as a direct consequence of \cite[Definition 2.11 and Corollary 2.12]{EK} we obtain the following theorem.

\begin{theorem} \label{teorema5deoutubro}
Let $\mathbb{K}$ be a field, let $\tilde{G}$ be an abelian group and let $R$ be a finite dimensional $\tilde{G}$-graded algebra 
over $\mathbb{K}$. If $R$ is graded simple then $R$ is isomorphic to some $\mathcal{M}(D, \kappa,\gamma)$ as $\tilde{G}$-graded algebras.

Two $\tilde{G}$-graded algebras $\mathcal{M}(D, \kappa,\gamma)$ and $\mathcal{M}(D', \kappa^{\prime},\gamma^{\prime})$ are isomorphic if and only if all of the following items are satisfied:
\begin{itemize}
\item[(a)] $D\cong D'$ as $\tilde{G}$-graded algebras, 
\item[(b)] $\kappa$ and $\kappa^{\prime}$ have the same number of components $s$,
\item[(c)] there exist an element $g\in \tilde{G}$, a permutation $\pi\in Sym(s)$, and $h_1$, \dots, $h_s\in\mathrm{Supp}\, D$ such that $k_i^{\prime}= k_{\pi(i)}$ and $g_i^{\prime}=g_{\pi(i)}h_ig$ for every $i$.
\end{itemize} 
\end{theorem}

We observe that \cite[Corollary 2.12]{EK} holds for arbitrary groups, however the hypothesis that ``$\tilde{G}$ is abelian"
will be used in the main result of the next section and the more general version of the above result would lead to a more complicated notation.

If the group $G$ in Theorem \ref{teorema16denovembro} is assumed to be abelian, we may use Theorem \ref{teorema5deoutubro} above to obtain a classification of the $G$-gradings on finite matrix algebras up to isomorphism. This will be done in the next theorem. In order to proceed we shall need the following notation.

\begin{notation}
 Let $G$ be an abelian group such that every finite subgroup of $G$ is cyclic. Henceforth we assume that for each finite subgroup $T$ of $G$ a generator $\emph{\textbf{t}}$ is fixed, so $T=\langle \emph{\textbf{t}} \rangle$. If 
 \[
|T|=r, \ 1\leq b \leq r-1 \quad  \mbox{and}  \quad \emph{gcd}(r,b)=1, 
\]
 we denote by $D_{T,b}$ the $G$-graded algebra $D_{T, \mathbb{E}_r,\psi^{(r,b)}}$. Here $\psi^{(r,b)}$ is defined as in (\ref{11dejaneiro}). Moreover if $\kappa$ and $\gamma$ are as in Notation \ref{notation22set} we denote by $\mathcal{M}(T,b,\kappa,\gamma)$ the algebra $\mathcal{M}(D_{T,b},\kappa,\gamma)$.   
\end{notation}

\begin{theorem}\label{classicisom}
Let $\mathbb{F}$ be a finite field and let $G$ be an abelian group such that every finite subgroup of $G$ is cyclic.
If $M_n(\mathbb{F})$ is a $G$-graded algebra, then it is isomorphic to some $\mathcal{M}(T,b,\kappa,\gamma)$. 

Moreover two $G$-graded algebras $\mathcal{M}(T,b,\kappa,\gamma)$ and $\mathcal{M}(T^{\prime},b^{\prime},\kappa^{\prime},\gamma^{\prime})$ are isomorphic if and only  if all of the following items are satisfied:
\begin{itemize}
\item[(a)] $T= T'$ and $b=b'$, 
\item[(b)] $\kappa$ and $\kappa^{\prime}$ have the same number of components $s$,
\item[(c)] there exist an element $g\in G$, a permutation $\pi\in Sym(s)$, and $h_1$, \dots, $h_s\in T$
such that $k_i^{\prime}= k_{\pi(i)}$ and $g_i^{\prime}=g_{\pi(i)}h_ig$ for every $i$.
\end{itemize}  
\end{theorem}
\begin{proof}
Theorem \ref{teorema16denovembro} and Theorem \ref{teorema5deoutubro} imply that $M_n(\mathbb{F})$ is isomorphic to some $\mathcal{M}(D_{T,\mathbb{E},\psi}, \kappa,\gamma)$. Theorem \ref{teorema2de16denovembro} implies that $D_{T,\mathbb{E},\psi}$ is isomorphic to some $D_{T,b}$. We apply Theorem \ref{teorema5deoutubro} to conclude that $M_n(\mathbb{F})$ is isomorphic to some $\mathcal{M}(T,b,\kappa,\gamma)$.

Theorem \ref{divisioniso} gives us that the algebras $D_{T,b}$ and $D_{T',b'}$ are isomorphic if and only if $T=T'$ and $b=b'$. Hence the claim about the algebras $\mathcal{M}(T,b,\kappa,\gamma)$ and $\mathcal{M}(T^{\prime},b^{\prime},\kappa^{\prime},\gamma^{\prime})$ follows from Theorem \ref{teorema5deoutubro}.
\end{proof}

In what follows, we again make some comments concerning the results obtained in \cite{BDW}. Theorem \ref{teorema16denovembro} and \cite[Theorem~5.1]{BDW} are, in general, incomparable. Indeed, \cite{BDW} treats an arbitrary field but only prime matrix sizes, whereas Theorem \ref{teorema16denovembro} considers an arbitrary matrix size $n$, but only finite fields and groups whose finite subgroups are all cyclic. When $n=p$ is prime and the field is finite, the two results do overlap. In this case, \cite{BDW} provides a stronger result, since it also covers gradings with non-cyclic support. This case is not vacuous over finite fields. For example, over $\mathbb{F}_3$, consider the matrices
\[
x=
\begin{pmatrix}
1 & 0\\
0 & -1
\end{pmatrix}
\qquad\text{and}\qquad
y=
\begin{pmatrix}
0 & 1\\
1 & 0
\end{pmatrix}.
\]
They satisfy
\[
x^2=y^2=1
\qquad\text{and}\qquad
xy=-yx.
\]
Moreover, the four homogeneous components
\[
\mathbb{F}_3\cdot 1,\qquad
\mathbb{F}_3\cdot x,\qquad
\mathbb{F}_3\cdot y,\qquad
\mathbb{F}_3\cdot xy
\]
consist entirely of invertible elements. Thus, they define a division grading on $M_2(\mathbb{F}_3)$ whose support is the Klein four-group, and hence is non-cyclic.
The hypothesis that every finite subgroup of $G$ be cyclic therefore plays a crucial role. It excludes non-cyclic supports, so that the cyclic algebra description in our Theorem \ref{DivisionMatrix} applies to every division grading that can occur.

\section{Graded Polynomial Identities in Finite Matrix Algebras}\label{gradedpoly}

In this section, $\mathbb{F}$ is a finite field with $q$ elements, $G$ is an abelian group such that every finite subgroup of $G$ is cyclic, and $R$ and $R'$ are $G$-graded matrix algebras over $\mathbb{F}$ satisfying the same $G$-graded polynomial identities. Our main goal here is to prove that $R$ and $R'$ are isomorphic as $G$-graded algebras. 	
From now on, by Theorem \ref{classicisom}, we assume $R=\mathcal{M}(T,b,\kappa,\gamma)$ and $R^{\prime}=\mathcal{M}(T^{\prime},b^{\prime},\kappa^{\prime},\gamma^{\prime})$, and we write 
\[
\kappa=(k_1,\dots, k_s),\gamma=(g_1,\dots, g_s),\kappa'=(k_1',\dots, k_{s'}') \  \mbox{and} \ \gamma'=(g_1',\dots, g_{s'}').
\]
We have $R=A\otimes_{\mathbb{F}} D=M_t(D)$ where $A=M_t(\mathbb{F})$ for $t=k_1+\cdots+k_s$ and $D=D_{T,b}$.  Henceforth we assume that $A$ is a $G$-graded algebra with respect to the elementary $G$-grading induced by the tuple  
\[
(\underbrace{g_1,\ldots ,g_1}_{k_1},\underbrace{g_2,\ldots ,g_2}_{k_2},\ldots ,\underbrace{g_s,\ldots ,g_s}_{k_s}).
\] 
Since the group $G$ is abelian we have 
\[
\deg(e_{i,j}\otimes d)= (\deg e_{i,j})(\deg d),
\]   
where $d$ is a homogeneous element of $D$.
It follows from (\ref{equacao2de5deoutubro}) and (\ref{equacao5deoutubro}) that 
 \begin{equation} \label{18deoutubrob}
     R_e=A_e\otimes_{\mathbb{F}} D_e.
 \end{equation}

 In order to prove the main result of this section it will be useful to consider certain ordinary identities of $R_e$ as graded identities of $R$, hence we need to establish some notation to better understand the neutral component $R_e$. We recall that 
 \begin{align}\label{isomde}
 D_e=(D_{T,b})_e=\mathbb{E}_rY_e\cong \mathbb{E}_r
 \end{align}
 is a field extension of $\mathbb{F}$ of degree $r=| T |$. Henceforth $r=| T | =[D_e:\mathbb{F}]$. In order to describe $A_e$ we denote 
 $$I_1=\{1,\dots, k_1\} \ \ \mbox{and} \ \ I_l=\{k_1+\cdots+k_{l-1}+1,\ldots ,k_1+\cdots+k_l\}$$ for $2\leq l \leq s$. Therefore we have  
 \[
\deg e_{i,j}=g_ag_b^{-1} \  \mbox{for} \  i\in I_a,  j\in I_b.
\]  
Note that $e_{i,j}\in A_e$ if and only if there exists $l$ such that $i$, $j\in I_l$. The set $\{e_{i,j}\mid i,j\in I_l\}$ spans a subalgebra of $A_e$ that we denote by $A_{e,l}$. We have 
\begin{equation}\label{reisom}
A_{e,l}\cong M_{k_l}(\mathbb{F}) \ \mbox{and} \ A_e=A_{e,1}\oplus \cdots \oplus A_{e,s}.
\end{equation}

We observe that (\ref{isomde}) and (\ref{reisom}) imply that $R_e\cong M_{k_1}(\mathbb{E}_r)\oplus \cdots \oplus M_{k_s}(\mathbb{E}_r)$. This clearly gives us the next result.

\begin{lemma}\label{idre}
$Id(R_e)=Id(M_m(\mathbb{E}_r))$, where $m=\max\{k_1,\dots, k_s\}$ and furthermore $r=| T | =[D_e:\mathbb{F}]$.
\end{lemma}

We use the analogous notation for $R^{\prime}$ and obtain the following statement as an immediate consequence of Lemma \ref{idre}.

\begin{proposition}\label{field}
$[D_e:\mathbb{F}]=[D_e^{\prime}:\mathbb{F}]$.
\end{proposition}

\begin{proof}
     Since $R$ and $R^{\prime}$ satisfy the same graded identities we conclude that $Id(R_e)=Id(R_e^{\prime})$, hence Lemma \ref{idre} implies that $Id(M_m(\mathbb{E}_r))=Id(M_{m^{\prime}}(\mathbb{E}_{r^{\prime}}))$. Proposition \ref{ordinarymain} and Theorem \ref{teoremaamitsurlevitzki5deoutubro} imply that $M_m(\mathbb{E}_r)\cong M_{m^{\prime}}(\mathbb{E}_{r^{\prime}})$ and $m=m'$, hence $\mathbb{E}_{r}\cong \mathbb{E}_{r^{\prime}}$. Thus $[D_e:\mathbb{F}]=r=r^{\prime}=[D_e^{\prime}:\mathbb{F}]$.
\end{proof}

Observe that one can deduce the fact that $\mathbb{E}_{r}\cong \mathbb{E}_{r^{\prime}}$ directly from the isomorphism of the matrix algebras $M_m(\mathbb{E}_r)\cong M_{m^{\prime}}(\mathbb{E}_{r^{\prime}})$, since $\mathbb{E}_{r}$ and $\mathbb{E}_{r^{\prime}}$ are the respective centres of these matrix algebras.

\begin{remark}\label{r}
We recall that $D=D_{T,b}=D_{T, \mathbb{E}_r,\psi^{(r,b)}}$, where $r=| T| =[D_e:\mathbb{F}]$, and $D^{\prime}=D_{T^{\prime},b^{\prime}}=D_{T^{\prime}, \mathbb{E}_{r^{\prime}},\psi^{(r^{\prime},b^{\prime})}}$, where $r^{\prime}=| T^{\prime}| =[D_e^{\prime}:\mathbb{F}]$. It follows from Proposition \ref{field} that $r=r^{\prime}$.
\end{remark}

\begin{lemma} \label{lema17deoutubroa}
	If $r=[D_e: \mathbb{F}]$ and $\displaystyle \alpha=\frac{q^r-1}{q-1}$, then 
 \[\begin{array}{ccc}
      D_e&\rightarrow  & \mathbb{F}\\
      \delta&\mapsto & \delta^{\alpha}
 \end{array}\]
 is an onto map.
\end{lemma}

\begin{proof}
    $D_e$ is a field and $D_e\setminus\{0\}$ is a cyclic multiplicative  group of order $q^{r}-1$, see for example \cite[Theorem 3.30]{R}. Denoting by $d$ its generator, we have that $d^{\alpha}$ is a generator of the cyclic subgroup $\mathbb{F}\setminus\{0\}$. Therefore $\delta\mapsto \delta^{\alpha}$ maps $D_e$ onto $\mathbb{F}$.
    \end{proof}

From now on we assume, by Theorem \ref{classicisom}, that $k_1\geq \cdots \geq k_s$ in $\kappa$. Let $m=\max\{k_1,\dots, k_s\}$, and 
let $w$ be the positive integer such that $k_1=\cdots =k_w=m$ and $k_i<m$ if $i>w$. 

In the next result we determine the support of $D$ by means of the $G$-graded polynomial identities of $R$. To this end, it is important to recall the Formanek polynomial $F_m$ from Theorem \ref{teoremadeformanek}. The following remark will be also useful in the proof of the next lemma.

\begin{remark}\label{reme}
The matrix $e_l=\sum_{i\in I_l}e_{i,i}$ is the identity of $A_{e,l}$, and a matrix $e_{i,j}\in A$ lies in $A_e$ if and only if $e_le_{i,j}\neq 0$ and $e_{i,j}e_l\neq 0$ for some $l$.
\end{remark}

\begin{lemma} \label{18deoutubroa}
Consider $R=\mathcal{M}(T,b,\kappa,\gamma)=M_t(\mathbb{F})\otimes D$ and $m=\max\{k_1,\dots, k_s\}$ as before. Denote by 
$F_m=F_m(x_{1,e},y_{1,e},\dots, y_{m,e})$ the Formanek polynomial and let $g\in G$. Then 
\[g\in \mathrm{Supp}\, D \ \mbox{if and only if} \  F_mx_{2,g}\notin Id_G(R)  \ \mbox{and} \  [F_m^{\alpha}, x_{2,g}]\in Id_G(R),\] where $\displaystyle \alpha=\frac{q^r-1}{q-1}$ and $r=[D_e: \mathbb{F}]$.
\end{lemma}

\begin{proof}
Let $F_m(R_e)$ be the set of all elements  
\[F_m(X_{1,e},Y_{1,e},\dots, Y_{m,e})\]
such that $X_{1,e},Y_{1,e},\dots, Y_{m,e}\in R_e$.
By Theorem \ref{teoremadeformanek} and \cite[Exercise 7.3.23]{D} we obtain 
    \begin{equation} \label{equacaodofm}
    F_m(R_e)=(e_1\otimes D_e)\oplus \cdots \oplus (e_w\otimes D_e).
    \end{equation}
    
    Let $g\in \mathrm{Supp}\, D$.  
     By (\ref{equacao2de5deoutubro}) and (\ref{equacao5deoutubro}) we have
 \[R_g=A_e\otimes_{\mathbb{F}} D_g.\]
The element $r_g=e_{1,1}\otimes d_g$, where $d_g$ is a non-zero element of $D_g$, lies in $R_g$. Moreover, $e_1\otimes 1\in F_m(R_e)$ and $(e_1\otimes 1)r_g=r_g\neq 0$, therefore $F_mx_{2,g}\notin Id_G(R)$. By Lemma \ref{lema17deoutubroa}, $$F_m(R_e)^{\alpha}=(e_1\otimes \mathbb{F})\oplus \cdots \oplus (e_w\otimes \mathbb{F}),$$
which commutes with every element in $R_g$. Therefore $[F_m^{\alpha}, x_{2,g}]\in Id_G(R)$. 

Now, let $g\in G$ be such that $F_mx_{2,g}\notin Id_G(R)$ and $[F_m^{\alpha}, x_{2,g}]\in Id_G(R)$. Then for some $e_{i,j}\otimes d_t\in R_g$, where $t\in \mathrm{Supp}\, D$ and $d_t\in D_t$, and some $(e_l\otimes 1)$ we have $(e_l\otimes 1)(e_{i,j}\otimes d_t)\neq 0$. We claim that $(e_{i,j}\otimes d_t)(e_l\otimes 1)\neq 0$. Indeed, assuming that $(e_{i,j}\otimes d_t)(e_l\otimes 1)= 0$, since $(e_l\otimes 1)\in F_m(R_e)^{\alpha}$ and 
\[
[(e_l\otimes 1),(e_{i,j}\otimes d_t)]=(e_l\otimes 1)(e_{i,j}\otimes d_t)\neq 0,
\] 
we conclude that $[F_m^{\alpha}, x_{2,g}]\notin Id_G(R)$, a contradiction. Hence,  $(e_l\otimes 1)(e_{i,j}\otimes d_t)\neq 0$ and $(e_{i,j}\otimes d_t)(e_l\otimes 1)\neq 0$ imply that $i$, $j\in I_l$ and therefore Remark \ref{reme} yields that $e_{i,j}\in A_e$. Since the group $G$ is abelian it follows that $e_{i,j}\otimes d_t$ is of degree $g=t\in \mathrm{Supp}\, D$. This concludes the proof.
\end{proof}

\begin{proposition} \label{proposicao17denovembro1}
Consider $R=\mathcal{M}(T,b,\kappa,\gamma)=M_t(\mathbb{F})\otimes D$ and $R^{\prime}=\mathcal{M}(T^{\prime},b^{\prime},\kappa^{\prime},\gamma^{\prime})=M_{t^{\prime}}(\mathbb{F})\otimes D^{\prime}$ as before, that is, satisfying $Id_G(R)=Id_G(R')$.  
Then $\kappa$ and $\kappa^{\prime}$ have the same number of components $s$, there exist an element $g\in G$, a permutation $\pi\in Sym(s)$ and $h_1$, \dots, $h_s\in\mathrm{Supp}\, D$ such that $k_i^{\prime}= k_{\pi(i)}$ and $g_i^{\prime}=g_{\pi(i)}h_ig$ for every $i$.
\end{proposition}

\begin{proof}
By Proposition \ref{field} and Lemma \ref{18deoutubroa}, we have 
\[\mathrm{Supp}\, D=\mathrm{Supp}\, D^{\prime}.\]

Let $p_1=k_1$ and $p_{i}=p_{i-1}+k_i$ for $i=2$, \dots, $s$. We denote by $f_{\kappa,\gamma}$ the polynomial
\begin{align*}
&s_{2k_1-1}(x_{1,e},\dots, x_{2p_1-1,e})\cdot x_{2p_1,g_{1,2}}\cdot s_{2k_2-1}(x_{2p_1+1,e},\dots, x_{2p_2-1,e})\cdot  x_{2p_2,g_{2,3}}\cdot \cdots\\& \cdot x_{2p_{s-1},g_{s-1,s}} \cdot s_{2k_s-1}(x_{2p_{s-1}+1,e},\dots, x_{2p_s-1,e}),
\end{align*}
where $g_{l,l+1}=g_lg_{l+1}^{-1}$ for $1\leq l \leq s-1$. Replacing the variables 
\[
x_{1,e},\ x_{2,e}, \ x_{3,e}, \ x_{4,e}, \ \ldots  \ \ \mbox{by} \ \ e_{1,1}\otimes 1, \ e_{1,2}\otimes 1, \ e_{2,2}\otimes 1, \ e_{2,3}\otimes 1, \ \ldots \] 
we see that $f_{\kappa,\gamma}\notin Id_G(R)$. 

We recall that $\kappa^{\prime}=(k_1^{\prime},\dots, k_{s^{\prime}}^{\prime})$, $I_1^{\prime}=\{1,\dots, k_1^{\prime}\}$, and 
\[
I_l^{\prime}=\{k_1^{\prime}+\cdots+k_{l-1}^{\prime}+1,k_1^{\prime}+\cdots+k_{l-1}^{\prime}+2,\dots,k_1^{\prime}+\cdots+k_{l-1}^{\prime}+k_l^{\prime}\}
\] 
for $2\leq l\leq  s^{\prime}$. As in (\ref{18deoutubrob}), 
\[
R_e^{\prime}=\mathcal{R}_1^{\prime}\oplus\cdots \oplus \mathcal{R}_{s^{\prime}}^{\prime},
\] 
where $\mathcal{R}_l^{\prime}$ is the subspace spanned by the elements $e_{i,j}\otimes \delta$ where $i$, $j\in I_l^{\prime}$ and $\delta \in D_e^{\prime}$. By hypothesis $f_{\kappa,\gamma}\notin Id_G(R')$, hence there exist $u_1$, \dots,  $u_s\in \{1,\dots, s^{\prime}\}$, $a_1\in s_{2k_1-1}(\mathcal{R}_{u_1}^{\prime})$, \dots, $a_s\in s_{2k_s-1}(\mathcal{R}_{u_s}^{\prime})$, and $b_1=e_{i_1,j_1}\otimes \delta_1\in R_{g_{1,2}}^{\prime}$, \dots, $b_{s-1}=e_{i_{s-1},j_{s-1}}\otimes \delta_{s-1}\in R_{g_{s-1,s}}^{\prime}$ such that 
\begin{equation}\label{9denovembro}
a_1b_1a_2b_2 \cdots a_{s-1}b_{s-1}a_s\neq 0.
\end{equation}

\noindent \emph{Claim 1.} $u_l\neq u_v$ whenever $l\neq v$.

\noindent \emph{Proof of Claim 1.} Without loss of generality we assume $l<v$. The element \[0\neq b_la_{l+1}b_{l+1}\cdots a_{v-1}b_{v-1}=e_{i_l,j_{v-1}}\otimes \delta\] is homogeneous of degree $g_lg_{v}^{-1}\notin \mathrm{Supp}\, D=\mathrm{Supp}\, D^{\prime}$. As a consequence the degree of $e_{i_{l},j_{v-1}}\otimes 1$ in $R^{\prime}$ is not equal to $e\in G$. Since $i_{l}\in I_{u_l}'$ and $j_{v}\in I_{u_v}'$ we conclude that $u_l\neq u_v$. Hence Claim 1 is proved.

\

\noindent \emph{Claim 2.} $\kappa$ and $\kappa^{\prime}$ have the same number of components $s$.

\noindent \emph{Proof of Claim 2.} Since $u_1$, \dots, $u_s\in \{1,\dots, s^{\prime}\}$, Claim 1 implies that $s\leq s^{\prime}$. Since $f_{\kappa^{\prime},\gamma^{\prime}} \notin Id_G(R)$,  an analogous argument implies that $s^{\prime}\leq s$. This proves the claim.

\

\noindent \emph{Claim 3.} There exists a permutation $\tau$ of the elements $\{1,\dots, s\}$ such that $k_i^{\prime}=k_{\tau(i)}$ for $i=1$, \dots, $s$.

\noindent \emph{Proof of Claim 3.} Claim 1 and Claim 2 imply that there exists a permutation $\tau$ of the elements $\{1,2,\dots, s\}$ such that $u_l=\tau(l)$ for $l=1$, \dots, $s$. By  (\ref{9denovembro}), the standard polynomial of degree $2k_l-1$ is not an identity for 
\[
\mathcal{R}_{\tau(l)}^{\prime}\cong M_{k_{\tau(l)}^{\prime}}(\mathbb{F})\otimes_{\mathbb{F}} D_e^{\prime},
\] 
therefore $k^{\prime}_{\tau(l)}\geq k_l$ for $l=1$, \dots, $s$. Since $f_{\kappa^{\prime},\gamma^{\prime}} \notin Id_G(R)$, there exists a permutation $\mu$ of the elements $\{1,2,\dots, s\}$ such that $k_{\mu(l)}\geq k_l^{\prime}$ for $l=1$, \dots, $s$. Taking into account the inequalities $k_1\geq \cdots \geq k_s$ and $k_1^{\prime}\geq \cdots \geq k_s^{\prime}$, this implies that $k_i=k_{\tau(i)}^{\prime}$ for $i=1$, \dots, $s$.

\

\noindent \emph{Claim 4.} There exist an element $g\in G$ and $h_1$, \dots,  $h_s\in\mathrm{Supp}\, D$ such that $g_i=g_{\tau(i)}^{\prime}h_ig$ for $i=1$,  \dots, $s$. 

\noindent \emph{Proof of Claim 4.} Recall that $a_lb_la_{l+1}\neq 0$, and $b_l=e_{i_l,j_l}\otimes \delta_l$ is homogeneous of degree $g_lg_{l+1}^{-1}$ in $R'$. Since $a_l\in \mathcal{R}_{\tau(l)}^{\prime}$ and $a_{l+1}\in \mathcal{R}_{\tau(l+1)}^{\prime}$ we conclude that $i_l\in I_{\tau(l)}^{\prime}$ and $j_l\in I_{\tau(l+1)}^{\prime}$, therefore 
\[
g_lg_{l+1}^{-1}=g_{\tau(l)}^{\prime}\deg (\delta_l)(g_{\tau(l+1)}^{\prime})^{-1}\] 
for $l=1$, \dots, $s-1$. Let $h_1$, \dots, $h_s$ be elements of $\mathrm{Supp}\, D^{\prime}$ such that $h_lh_{l+1}^{-1}=\mathrm{deg}\, (\delta_l)$ for $l=1$,  \dots, $s-1$. Then we have 
\[
g_lg_{l+1}^{-1}=g_{\tau(l)}^{\prime}h_lh_{l+1}^{-1}(g_{\tau(l+1)}^{\prime})^{-1},\]
and therefore 
\[
h_l^{-1}(g_{\tau(l)}^{\prime})^{-1}g_l=
h_{l+1}^{-1}(g_{\tau(l+1)}^{\prime})^{-1}g_{l+1}.
\] 
for $l=1$, \dots, $s-1$. If $g=h_1^{-1}(g_{\tau(1)}^{\prime})^{-1}g_1$, the equality above implies that $g_i=g_{\tau(i)}^{\prime}h_ig$ for $i=1$, \dots, $s$ as required.
\end{proof}

\begin{theorem}\label{main}
Let $\mathbb{F}$ be a finite field with $q$ elements, $G$ an abelian group such that every finite subgroup of $G$ is cyclic, $R=M_n(\mathbb{F})$ and $R'=M_{n'}(\mathbb{F})$ two $G$-graded matrix algebras over $\mathbb{F}$. Then $R$ and $R'$ are isomorphic as $G$-graded algebras if and only if $Id_G(R)=Id_G(R')$.
\end{theorem}

\begin{proof}
It is obvious that two isomorphic graded algebras satisfy the same graded identities.

Suppose that $Id_G(R)=Id_G(R')$. By Theorem \ref{classicisom}, we can assume $R=\mathcal{M}(T,b,\kappa,\gamma)=M_t(\mathbb{F})\otimes D$ and $R^{\prime}=\mathcal{M}(T^{\prime},b^{\prime},\kappa^{\prime},\gamma^{\prime})=M_{t^{\prime}}(\mathbb{F})\otimes D^{\prime}$ as before. By Proposition \ref{field} and Lemma \ref{18deoutubroa}, we have 
\begin{align}\label{eqsup}
T=\mathrm{Supp}\, D=\mathrm{Supp}\, D^{\prime}=T'.
\end{align}
The above equality, together with Remark \ref{r} imply that $D=D_{T,b}=D_{T, \mathbb{E}_r,\psi^{(r,b)}}$ and $D^{\prime}=D_{T^{\prime},b^{\prime}}=D_{T,b^{\prime}}=D_{T, \mathbb{E}_r,\psi^{(r,b^{\prime})}}$. In particular $D=D^{\prime}$ if $b=b^{\prime}$.

\


\noindent \textit{Claim 1.} $b=b^{\prime}$ and therefore $D=D^{\prime}$.

\noindent \textit{Proof of Claim 1}.
The result follows if we prove that $\psi^{(r,b)}_{\textbf{t}}=\psi^{(r,b^{\prime})}_{\textbf{t}}$, where $\textbf{t}$ is a generator of $T$. By the comments before Theorem \ref{teorema2de16denovembro}, there exist integers $l$, $u\geq 0$ such that  
\[
\psi^{(r,b)}_{\textbf{t}}(x)=x^{p^l} \ \ \mbox{and} \ \ \psi^{(r,b^{\prime})}_{\textbf{t}}(x)=x^{p^u},
\]
for every $x\in \mathbb{E}_r$. Given an integer $v\geq 0$, by (\ref{equacaodofm}) we obtain that
\[
\left(x_{2,\textbf{t}}x_{1,e}-x_{1,e}^{p^v}x_{2,\textbf{t}}\right)\in Id_G(D) \ \mbox{if and only if} \ \left( x_{2,\textbf{t}}F_m-(F_m)^{p^v}x_{2,\textbf{t}} \right) \in Id_G(R).
\] 
Here, $F_m=F_m(x_{1,e},y_{1,e},\ldots, y_{m,e})$ is the  Formanek polynomial.
The analogous remark holds for the algebras $D^{\prime}$ and $R^{\prime}$. Hence 
\[
\left(x_{2,\textbf{t}}x_{1,e}-x_{1,e}^{p^v}x_{2,\textbf{t}}\right)\in Id_G(D) \ \mbox{if and only if} \ \left(x_{2,\textbf{t}}x_{1,e}-x_{1,e}^{p^v}x_{2,\textbf{t}}\right)\in Id_G(D').
\]
Since $\left(x_{2,\textbf{t}}x_{1,e}-x_{1,e}^{p^l}x_{2,\textbf{t}}\right)$ and $\left(x_{2,\textbf{t}}x_{1,e}-x_{1,e}^{p^u}x_{2,\textbf{t}}\right)$ are $G$-graded identities for $D$ and $D^{\prime}$, respectively, as a consequence $x^{p^l}=x^{p^u}$ for every $x\in \mathbb{E}_r$. Hence $\psi^{(r,b)}_{\textbf{t}}=\psi^{(r,b^{\prime})}_{\textbf{t}}$ as desired.

\

From Equality (\ref{eqsup}), Claim 1, Proposition \ref{proposicao17denovembro1},  and Theorem \ref{classicisom}, we obtain that $R$ and $R'$ are isomorphic 
as $G$-graded algebras. This concludes the proof.
\end{proof}


\end{document}